\documentclass[12pt]{amsart}
\usepackage{amssymb}

\usepackage{mathtools}
\usepackage{bbm}
\usepackage{color}
\allowdisplaybreaks

\theoremstyle{definition}
\newtheorem{theorem}{Theorem}
\newtheorem{definition}[theorem]{Definition}  
\newtheorem{lemma}[theorem]{Lemma}
\newtheorem{notation}[theorem]{Notation}

\newtheorem{corollary}[theorem]{Corollary}

\newtheorem{proposition}[theorem]{Proposition}

\newcommand{\az}{\textrm{A}$0$}
\newcommand{\ao}{\textrm{A}$1$}
\newcommand{\at}{\textrm{A}$2$}
\newcommand{\ab}{\allowbreak}
\newcommand{\bC}{\mathbb{C}}
\newcommand{\bR}{\mathbb{R}}
\newcommand{\cA}{\mathcal{A}}
\newcommand{\cK}{\mathcal{K}}
\newcommand{\cS}{\mathcal{S}}
\newcommand{\E}{\mathrm{E}}

\newcounter{jmpnumber}

\newcounter{int}
\newcommand{\listcomments}{%
\smallskip\hrule\smallskip\noindent 
\@whilenum\value{int}<\thejmpnumber\do
{\tiny Comm. \theint{} is on page \pageref{jmp\theint}\stepcounter{int}, }
\smallskip\hrule\smallskip}

\makeatletter
\newcommand{\myref}[1]{$(\expandafter\@alph{\ref{#1}})$}
\makeatother
\newcommand{\centre}{\mathaccent"7017}

\begin{document}

\title[Analytic Structure of Second Order Probability
  Spaces]{On the Analytic Structure of Second-Order
  Non-Commutative Probability Spaces and Functions of
  Bounded Fr\'echet Variation}

\author[Diaz]{MARIO DIAZ}

\address{Instituto de Investigaciones en Matem\'{a}ticas
  Aplicadas y en Sistemas, Universidad Nacional Aut\'{o}noma
  de M\'{e}xico, Mexico City, Mexico}
\email{mario.diaz@sigma.iimas.unam.mx}

\author[Mingo]{JAMES A.~MINGO}

\address{Department of Mathematics and Statistics, Queen's
  University, Jeffery Hall, Kingston, Ontario, K7L 3N6,
  Canada} \email{mingo@mast.queensu.ca}

\thanks{This work was supported in part by a
    Discovery Grant from the Natural Sciences and
    Engineering Research Council of Canada, and in part by
    the Consejo Nacional de Ciencia y Tecnolog\'{i}a of
    Mexico under Grant A1-S-976.}

\begin{abstract}
In this paper we propose a new approach to the central limit
theorem (CLT), based on functions of bounded F\'echet
variation for the continuously differentiable linear
statistics of random matrix ensembles which
relies on: a weaker form of a large deviation
principle for the operator norm; a Poincar\'{e}-type
inequality for the linear statistics; and the existence of a
second-order limit distribution. This approach frames into a
single setting many known random matrix ensembles and, as a
consequence, classical central limit theorems for linear
statistics are recovered and new ones are established, e.g.,
the CLT for the continuously differentiable linear
statistics of block Gaussian matrices.

In addition, our main results contribute to the
understanding of the analytical structure of second-order
non-commutative probability spaces. On the one hand, they
pinpoint the source of the unbounded nature of the bilinear
functional associated to these spaces; on the other hand,
they lead to a general archetype for the integral
representation of the second-order Cauchy transform,
$G_2$. Furthermore, we establish that the covariance of
resolvents converges to this transform and that the limiting
covariance of analytic linear statistics can be expressed as
a contour integral in $G_2$.
\end{abstract}

\maketitle

\section{Introduction}

In his seminal work \cite{Wigner1955,Wigner1957,Wigner1958},
Wigner established that the empirical spectral measure of
certain random matrix ensembles converges, as the dimension
goes to infinity, to the semicircle distribution. Since
then, several other asymptotic phenomena have been
discovered for a wide range of random matrix
ensembles. Examples of these phenomena include, but are not
limited to, large deviations for several spectral objects
\cite{BenArousGuionnet1997,Maida2007,LedouxRider2010};
convergence and strong convergence of the empirical spectral
distribution \cite{MarchenkoPastur1967,
  HaagerupThorbjornsen2005, Male2012, Anderson2013,
  CollinsMale2014}; asymptotic non-commutative independence
between random matrix ensembles \cite{Voiculescu1991,
  Voiculescu1998, AndersonFarrell2014, MingoPopa2016}, and
central limit theorems for linear statistics
\cite{DiaconisShahshahani1994, Johansson1998,
  BaiSilverstein2004, AndersonZeitouni2006,
  PasturShcherbina2011}.

Another such phenomenon, which stems mainly from
combinatorial considerations, is the fact that many random
matrix ensembles have a second-order limit distribution
\cite{MingoSpeicher2006, MingoSniadySpeicher2007,
  CollinsMingoSniadySpeicher2007}. We say that a random
matrix ensemble $(X_N)_{N\in\mathbb{N}}$ has a second-order
limit distribution if
\begin{itemize}

\item[i)] For all $m,n\in\mathbb{N}$, the following limits
  exist

\begin{equation*}
\alpha_n = \lim_{N\to\infty} \frac{1}{N}
\mathbb{E}(\textnormal{Tr}(X_N^n)) \quad \text{ and } \quad
\alpha_{m,n} = \lim_{N\to\infty}
\textnormal{Cov}(\textnormal{Tr}(X_N^m),\textnormal{Tr}(X_N^n)),
\end{equation*}
where $\textnormal{Cov}(X,Y) = \mathbb{E}(XY) -
\mathbb{E}(X)\mathbb{E}(Y)$;

\item[ii)]
For all $r\geq3$ and all
  $n_1,\ldots,n_r\in\mathbb{N}$,
\begin{equation*}
\lim_{N\to\infty} k_r\left(\textnormal{Tr} (X_N^{n_1}),
\ldots, \textnormal{Tr} (X_N^{n_r})\right) = 0,
\end{equation*}
where $k_r$ denotes the classical cumulant of order $r$.

\end{itemize}
In this paper we propose a new approach to the central limit
theorem (CLT) for continuously differentiable linear
statistics which is based on three other phenomena: a weaker
form of a large deviation principle for the operator norm; a
Poincar\'{e}-type inequality for the linear statistics; and
the existence of a second-order limit distribution. The
first two phenomena ensure the existence of the limiting
covariance of linear statistics, while the latter leads to
their asymptotic Gaussianity. Apart from making explicit the
relations between different fundamental phenomena, this
approach frames into a single setting many known random
matrix ensembles. As a consequence, classical central limit
theorems for linear statistics are recovered and new ones
are established, e.g., the CLT for the continuously
differentiable linear statistics of block Gaussian matrices.

Let $\lambda_1,\ldots,\lambda_N$ be the eigenvalues of an
$N\times N$ random matrix $X_N$. Given a function
$f:\mathbb{C}\to\mathbb{C}$, the (random) quantity
$\textnormal{Tr}(f(X_N)) = \sum_k f(\lambda_k)$ is called a
linear statistic of $X_N$. Hence, a CLT for the linear
statistics of a random matrix ensemble is a result that
establishes that, as $N\to\infty$,
\begin{equation*}
	\textnormal{Tr}(f(X_N)) -
        \mathbb{E}(\textnormal{Tr}(f(X_N))) \Rightarrow
        \mathcal{N}_{\mathbb{R}}(0,\sigma_f^2),
\end{equation*}
for some $\sigma_f^2>0$, where $\Rightarrow$
  denotes convergence in distribution and
  $\mathcal{N}_{\mathbb{R}}(\mu,\sigma^{2})$ denotes the
  real Gaussian distribution with mean $\mu$ and variance
  $\sigma^{2}$. If the family of functions $f$ for which
such a CLT holds is a vector space, then any $n$-tuple of
(centered) linear statistics is, in the limit, jointly
Gaussian. As a result, the covariance mapping
\begin{equation}
\label{eq:Introrho}
	\langle f,g\rangle \mapsto \lim_{N\to\infty}
        \textnormal{Cov}(\textnormal{Tr}(f(X_N)),\textnormal{Tr}(g(X_N)))
\end{equation}
plays a privileged role in the description of the asymptotic
behavior of the linear statistics. Indeed, the so-called
second-order non-commutative probability spaces capture this
covariance mapping in a bilinear functional, which is not
present in typical non-commutative probability spaces.

A second-order non-commutative probability space is a triple
$(\mathcal{A},\varphi, \ab\rho)$ consisting of a unital
algebra $\mathcal{A}$, a unital linear functional
$\varphi:\mathcal{A}\to\mathbb{C}$ which is tracial, and a
bilinear functional
$\rho:\mathcal{A}\times\mathcal{A}\to\mathbb{C}$ which is
tracial in both arguments and $\rho(1,a)=\rho(a,1)=0$ for
all $a\in\mathcal{A}$. The canonical example in the single
random matrix setting is the following. Let $(X_N)_N$ be a
self-adjoint random matrix ensemble. Define
$\mathcal{A}=\mathbb{C}[x]$ and, for all
$p,q\in\mathcal{A}$,
\begin{align*}
	\varphi(p) &= \lim_{N\to\infty} \frac{1}{N}
        \mathbb{E}(\textnormal{Tr}(p(X_N))) \quad \text{ and
        } \\ \rho(p,q) &= \lim_{N\to\infty}
        \textnormal{Cov}(\textnormal{Tr}(p(X_N)),\textnormal{Tr}(q(X_N))).
\end{align*} 
We will assume that $\cA = C([-M, M])$ and that the linear
functional $\varphi$ is continuous with respect to the
supremum norm on $C([-M,M])$, the space of continuous
functions on $[-M, M]$, for some $M>0$. However, it was
observed that in some canonical examples, the bilinear
functional $\rho$ is not continuous (bounded) with respect
to the the supremum norm on $C([-M,M])$ for any
$M>0$. (Theorem~\ref{Theorem:BilinearFunctional} below shows
that in fact this is always the case.) The unboundedness of
$\rho$ makes the usual analytic setting from free
probability theory unfitted for the second-order case. In
addition to the CLT for linear statistics, our main results
contribute to the understanding of the analytic structure of
second-order non-commutative probability spaces. On the one
hand, they pinpoint the source of the unbounded nature of
the bilinear functional $\rho$; on the other hand, they lead
to a general archetype for the integral representation of
the second-order Cauchy transform, i.e., the generating
function given by
\begin{equation}
\label{eq:IntroSOCauchyTransform}
	G_2(z,w) = \sum_{m,n\geq 1}
        \frac{\alpha_{m,n}}{z^{m+1}w^{n+1}}.
\end{equation}
Furthermore, we establish that the covariance of resolvents
converges to this transform (see Corollary
\ref{corollary:resolvents}) and that the limiting covariance
of analytic linear statistics can be expressed as a contour
integral depending on the second order Cauchy transform, see
Theorem \ref{Theorem:FluctuationsLinearStatistics}. Since
the second-order Cauchy transform of block Gaussian matrices
was recently found in \cite{diaz2020global}, these results
provide an effective way to compute the covariance of
analytic linear statistics of block Gaussian matrices.

The organization of this paper is as follows. We
present some preliminaries in the following
  section. We precisely state our main results on the CLT
  for the linear statistics of some random matrices in
  Section~\ref{Section:CLTLinearStatistics}, and provide
  their proofs in Section~\ref{Section:DeferredProofs}. In
Section~\ref{Section:Examples}, we provide some examples of
random matrix ensembles for which our main results apply. In
particular, we show that block Gaussian matrices fall within
the scope of our work. We introduce the
  definition of analytic second-order non-commutative
  probability space and discuss its implications in
  Section~\ref{Section:AnalyticsSONCPS}. We finish this
paper with some concluding remarks in
Section~\ref{Section:ConcludingRemarks}.

\section{Preliminaries}

\begin{notation}\label{notation:our_space_of_functions}
We let $C^1(\bR)$ denote the space of complex valued
functions on $\bR$ which have a continuous derivative. We
denote by $f\vert_{M}$ the restriction of
$f:\mathbb{R}\to\mathbb{C}$ to the interval $[-M,M]$ and we
let $C^1([-M, M]) = \{ f|_M \mid f \in C^1(\bR) \}$. For
$f\in C^1([-M,M])$, we let $\centre{f}(x) = f(x) - f(0)$ and
$C^1([-M, M])^\circ = \{ \centre f \mid f \in C^1([-M, M])
\}$. Then we have a direct sum decomposition
\[
C^1([-M, M]) = \bC \oplus C^1([-M, M])^\circ
\] 
and let $\| f \| = |f(0)| + \| f' \|_M$ where $\| f \|_M =
\sup_{|x| \leq M} |f(x)|$. This is a norm on $C^1([-M, M])$
and in this norm $C^1([-M, M])$ is a Banach space.
\end{notation}

\subsection{Fr\'{e}chet Representation Theorem for Bilinear Functionals}
\label{Section:FrechetIntegral}

Given a function $u:[-M,M]^{2}\to\mathbb{R}$ and
$s_{0},s_{1},t_{0},t_{1}\in[-M,M]$, we let
\begin{equation*}
	\Delta u (s_{0},s_{1};t_{0},t_{1}) = u(s_{1},t_{1})
        - u(s_{1},t_{0}) - u(s_{0},t_{1}) + u(s_{0},t_{0}).
\end{equation*}

\begin{definition}
We say $u:[-M,M]^{2}\to\mathbb{R}$ has \textit{bounded
  Fr\'{e}chet variation} if
\begin{equation*}
	\sum_{i=1}^{m} \sum_{j=1}^{n} \sigma_{i} \theta_{j}
        \Delta u(s_{i-1},s_{i};t_{j-1},t_{j})
\end{equation*}
is uniformly bounded for all $m,n\in\mathbb{N}$,
$\sigma_{i},\theta_{j}\in\{\pm1\}$, for all $ -M=s_{0} \leq
s_{1} \leq \cdots \leq s_{m} = M$, and all $ -M=t_{0} \leq
t_{1} \leq \cdots \leq t_{n} = M$.  If $u:[-M,M]^{2}\to\bC$
has real and imaginary parts which are of bounded Fr\'echet
variation, then we say that the complex valued function has
bounded Fr\'echet variation.
\end{definition}

\begin{definition}\label{definition:bounded_bilinear}
Let $\phi : C([-M, M])^2 \rightarrow \bC$ be a bilinear
function; we say that $\Phi$ is \textit{bounded} if
$\exists\, K > 0$ such that for all $f, g \in C([-M, M])$ we
have $|\Phi(f, g)| \leq K \| f \|_M\, \| g \|_M$, where
$\|f\|_M = \sup \{ |f(x)| : x\in[-M,M]\}$.
\end{definition}

The two results of Fr\'echet we need can be found in
\cite[Ch. III \S7 and \S11]{Hildebrandt1963} and \cite[\S
  6]{Frechet1915}. See also \cite{MorseTransue1949a} and
\cite{MorseTransue1949b}. For a partition $s =
  (s_{0},s_{1},\ldots,s_{n})$ of $[-M,M]$, i.e., $-M = s_{0}
  \leq s_{1} \leq \cdots \leq s_{n} = M$, we let $\lvert s
  \rvert = \max \{s_{i} - s_{i-1} : 1 \leq i \leq n\}$.

\begin{theorem}[Fr\'{e}chet]
If $u:[-M,M]^{2}\to\mathbb{R}$ has bounded Fr\'{e}chet
variation, then for all $f, g \in C([-M,M])$, the following
limit exists
\begin{equation*}
\Phi(f,g) := \lim_{\lvert s \rvert, \lvert t \rvert \to 0}
\sum_{i} \sum_{j} f(\xi_{i}) g(\eta_{j}) \Delta
u(s_{i-1},s_{i};t_{j-1},t_{j}),
\end{equation*}
where $\xi_{i} \in [s_{i-1}, s_i]$ and $\eta_j \in [t_{j-1},
  t_j]$.  Moreover there is $K \geq 0$ such that
$|\Phi(f,g)| \leq K \lVert f \rVert_M \lVert g \rVert_M$.
\end{theorem}

Given its similarities with the Riemann-Stieltjes integral
for functions of bounded variation, the limit $\Phi(f,g)$ in
the previous theorem is often called the \emph{Fr\'{e}chet
integral} of $f$ and $g$ with respect to $u$ and it is
denoted by
\begin{equation*}
	\int_{-M}^{M} \int_{-M}^{M} f(x) g(y) \mathrm{d} u(x,y).
\end{equation*}
Observe that if $u$ has bounded variation, then it also has
bounded Fr\'{e}chet variation and, for continuous functions
$f$ and $g$, both the Riemann-Stieltjes and Fr\'{e}chet
integrals coincide. In words, the previous theorem
establishes that the Fr\'{e}chet integral exists and
describes a bounded bilinear functional. The following
theorem establishes that the converse is also true: a
bounded bilinear functional is the Fr\'{e}chet integral with
respect to some function of bounded Fr\'{e}chet variation.

\begin{theorem}[Fr\'{e}chet]\label{thm:frechet_II}
If $\Phi:C([-M,M])^{2}\to\mathbb{R}$ is a bounded bilinear
functional, then there exists a function
$u:[-M,M]^{2}\to\mathbb{R}$ of bounded Fr\'{e}chet variation
such that
\begin{equation*}
	\Phi(f,g) = \int_{-M}^{M} \int_{-M}^{M} f(x) g(y)\,
        \mathrm du(x,y).
\end{equation*}
\end{theorem}

\section{CLT for the Linear Statistics of Some Random Matrices}
\label{Section:CLTLinearStatistics}

Let $(X_{N})_{N}$ be a random matrix ensemble. For all
$m,n\in\mathbb{N}$, we let
\begin{equation*}
	\alpha_{n} \coloneqq \lim_{N\to\infty} \frac{1}{N}
        \mathbb{E}\left(\textnormal{Tr}(X_{N}^{n})\right)
        \text{ and } \alpha_{m,n} \coloneqq
        \lim_{N\to\infty}
        \textnormal{Cov}\left(\textnormal{Tr}(X_{N}^{m}),\textnormal{Tr}(X_{N}^{n})\right).
\end{equation*}
The collections $(\alpha_n)_n$ and $(\alpha_{m,n})_{m,n}$
are the first and second-order moments of $(X_N)_N$,
respectively. In this paper we consider the following set of
assumptions.
\begin{itemize}
	\item[\textbf{A0.}] Both the first and second-order
          moments exist.
	\item[\textbf{A1.}] There exists $M>0$ such that $\displaystyle \lim_{N\to\infty} N^8 \Pr(\lVert X_N \rVert > M) = 0$.
	\item[\textbf{A2.}] There exists $K>0$ such that, for all $f\in C^1(\mathbb{R})$ and $N\in\mathbb{N}$,
	\begin{equation*}
		\textnormal{Var}\left(\textnormal{Tr}(f(X_N))\right)\leq K \|f'\|_\infty^2.
	\end{equation*}
\end{itemize}
Observe that A0 corresponds to Part i) in the definition of
second-order limit distribution. Also, observe that A1 is
weaker than a large deviation principle for the operator
norm. Finally, note that A2 is a matricial version of the
Poincar\'{e} inequality, see, e.g., Proposition~4.1 in
\cite{HaagerupThorbjornsen2005} and references therein. In
Section~\ref{Section:Examples} we show that these
assumptions are satisfied by some random matrix ensembles
which are common in the literature.

We define the set of smooth (test) functions $\mathcal{S}$ as
\begin{equation*}
	\mathcal{S} = \left\{f\in C^{1}(\mathbb{R}) : f
        \text{ is polynomially bounded}\right\}.
\end{equation*}
Note that assumption \az{} gives that we can define a
bi-linear map $\rho : \bC[x] \times \bC[x] \rightarrow \bC$
by linearly extending the definition $\rho(x^m, x^n) :=
\alpha_{m,n}$.

In the following theorem we show that under assumptions A0,
A1, and A2, $\rho$ extends to a bilinear function on
$C^1([-M, M])$, and the covariance of smooth linear
statistics converges, as $N \rightarrow \infty$, to $\rho$,
and is bounded by $K$, where the $K$ is that of assumption
A2.

This shows that the two ways one might define $\rho(f, g)$,
the first using fluctuation moments $(a)$
below, the second using linear statistics
$(b)$ below, agree. This is the main
technical part of the paper.

\begin{theorem}
\label{Theorem:BilinearFunctional}
If $(X_{N})_{N}$ is a self-adjoint random matrix ensemble
satisfying A0, A1, and A2, and $\rho$ the linear extension
of the second order moments in \az{} to $\bC[x]$, then,

\begin{enumerate}

\item\label{item:rho_extends} $\rho$ extends to a bi-linear
  function on $C^1([-M, M]) \times C^1([-M, M])$ satisfying
\begin{equation}\label{eq:BoundednessRho}
	|\rho(f,g)| \leq K \| f'\|_M \| g' \|_M;
\end{equation}

\item\label{item:covariance_limit}
for $f$ and $g$ in $\cS$
\begin{equation}\label{eq:covariance_limit}
	 \lim_{N\to\infty}
         \textnormal{Cov}(\textnormal{Tr}(f(X_N)),\textnormal{Tr}(g(X_N)))
         = \rho(f|_M, g|_M)
\end{equation}
where $f|_M$ denotes the restriction of $f$ to $[-M, M]$;

\item\label{item:integral_rep} there exists $u : [-M, M]^2
  \rightarrow \bR$, of bounded Fr\'{e}chet variation, such
  that
\[
\rho(f, g) 
=
\int_{-M}^{M} \int_{-M}^{M} f'(x) g'(y) \, \mathrm du(x,y),
\]
where the double integral is in the sense of Fr\'{e}chet as
discussed in Section~\ref{Section:FrechetIntegral}.
\end{enumerate}

\end{theorem}

By $(a)$, $\rho: C^1([-M, M]) \times C^1([-M, M])
\to\mathbb{C}$ is a well-defined bounded bilinear functional
which satisfies an asymptotic version of A2.  Indeed,
observe that A2 is equivalent to requiring that, for all
$f,g\in C^1(\mathbb{R})$ and $N\in\mathbb{N}$,
\begin{equation*}
	\lvert\textnormal{Cov}(\textnormal{Tr}(f(X_N)),\textnormal{Tr}(g(X_N)))\rvert
        \leq K \|f'\|_\infty \|g'\|_\infty.
\end{equation*}
In the presence of \ao{} we get the stronger statement
$(b)$.  The boundedness of $\rho$,
as established in \eqref{eq:BoundednessRho}, motivates
Definition \ref{def:soncps} of an analytic second-order NCPS
in Section~\ref{Section:AnalyticsSONCPS}.

\begin{notation}\label{notation:resolvents}
For each $z\in\mathbb{C}\setminus\mathbb{R}$, let
$r_{z}:\mathbb{R}\to\mathbb{C}$ be the function defined by
\begin{equation*}
	r_{z}(x) = \frac{1}{z-x}.
\end{equation*}
\end{notation}

Observe that $r_{z}$ is a smooth function, i.e.,
$r_{z}\in\mathcal{S}$. Therefore,
Theorem~\ref{Theorem:BilinearFunctional} implies that
$\rho(r_{z},r_{w})$ is well-defined for every
$z,w\in\mathbb{C}\setminus\mathbb{R}$ and gives us the
following corollary.

\begin{corollary}\label{corollary:resolvents}
For every $z, w \in \mathbb{C}\setminus\mathbb{R}$,
\begin{equation*}
	\rho(r_{z},r_{w}) = \lim_{N\to\infty}
        \textnormal{Cov}\left(\textnormal{Tr}((z-X_N)^{-1}),\textnormal{Tr}((w-X_N)^{-1})\right).
\end{equation*}
\end{corollary}
Since $\alpha_{m,n} = \rho(x^{m},x^{n})$ for all $m,n\geq1$,
\eqref{eq:BoundednessRho} readily implies that
\begin{equation*}
	|\alpha_{m,n}| \leq KmnM^{m+n-2}.
\end{equation*}
Thus the power series
\begin{equation*}
	 \sum_{m,n\geq 1} \frac{\alpha_{m,n}}{z^{m+1}w^{n+1}}
\end{equation*}
determines an analytic function on $\{(z,w)\in\mathbb{C}^2 : |z|,|w|>M\}$.

\begin{definition}\label{def:G_2}
Let $z, w \in \bC$ with $|z|, |w| > M$.We let
\begin{equation}\label{eq:def_G_2}
G_2(z, w) = \sum_{m,n \geq 1} \frac{\alpha_{m,m}}{z^{m+1} w^{n + 1}} 
\end{equation}
We call $G_2$ the \textit{second order Cauchy transform} of
the moment sequence $\{\alpha_{m,m}\}_{m, n \geq 1}$. $G_2$
is analytic on $\{(z,w)\in\mathbb{C}^2 : |z|,|w|>M\}$.
\end{definition}

The next theorem shows that $G_2$ can be analytically
extended to $(\mathbb{C}\setminus[-M,M])^2$ and that
$G_{2}(z,w)$ coincides with $\rho(r_{z},r_{w})$ for
$z,w\in\mathbb{C}\setminus\mathbb{R}$. As a result, the
following theorem establishes a connection between the limit
of the covariance of resolvents and the second-order Cauchy
transform.

\begin{theorem}
\label{Theorem:ConvergenceSOCauchyTransform}
If $(X_N)_N$ is a self-adjoint random matrix ensemble
satisfying A0, A1, and A2, then $G_2$ can be analytically
extended to $(\mathbb{C}\setminus[-M,M])^2$ and, for all
$z,w\in\mathbb{C}\setminus\mathbb{R}$,
\begin{equation*}
	G_2(z,w) = \rho(r_{z},r_{w}).
\end{equation*}
\end{theorem}

In \cite{diaz2020global}, Diaz et al.\ recently found a
formula for the second-order Cauchy transform of block
Gaussian matrices. Specifically, they derive a formula at
the level of formal expressions and then extend it to the
analytic level. Since block Gaussian matrices satisfy A0,
A1, and A2 (see
Section~\ref{Section:ExampleBlockGaussianMatrices} below),
Theorem~\ref{Theorem:ConvergenceSOCauchyTransform} implies
that their formula for the second-order Cauchy transform is
indeed equal to the limit of the covariance of resolvents.

The next theorem provides a formula for the limit of the
covariance of certain functions in terms of the second-order
Cauchy transform. Given the tools available to compute the
latter, the following theorem provides an effective way to
evaluate the asymptotic covariance mapping $\rho$.

\begin{theorem}
\label{Theorem:FluctuationsLinearStatistics}
Assume that $f,g\in\mathcal{S}$ satisfy that $f\vert_{M}$
and $g\vert_{M}$ extend analytically to a complex domain
$\Omega \supset [-M,M]$. If $(X_N)_N$ is a self-adjoint
random matrix ensemble satisfying A0, A1, and A2, then
\begin{equation}
\label{eq:ConvergenceLinearStatistics}
	\rho(f,g) = \frac{1}{(2\pi i)^2} \int_{\mathcal{C}}
        \int_{\mathcal{C}} f(z) g(w) G_{2}(z,w) \mathrm{d} z
        \mathrm{d} w,
\end{equation}
where $\mathcal{C}\subset\Omega$ is a positively oriented
simple closed contour enclosing $[-M,M]$.
\end{theorem}

We end this section with a central limit theorem for the
linear statistics of random matrix ensembles having a
second-order limit distribution. The statement of following
proposition is similar to Proposition~3.2.9 in
\cite{PasturShcherbina2011}. As the proof is different, we
have provided for the reader's convenience a proof in
Appendix~\ref{Appendix:ProofPropositionCLT} using the
notation and techniques from \S
\ref{Section:DeferredProofs}.

\begin{proposition}
\label{Proposition:CLT}
If $(X_N)_N$ is a self-adjoint random matrix ensemble having
a second-order limit distribution and satisfying A1 and A2,
then, for all $f\in\mathcal{S}$ with
$f(\mathbb{R})\subset\mathbb{R}$,
\begin{equation*}
	\textnormal{Tr}(f(X_N)) - \mathbb{E}(\textnormal{Tr}(f(X_N))) \Rightarrow
	{\mathcal N}_\mathbb{R}(0,\rho(f,f)),
\end{equation*}
where $\Rightarrow$ denotes convergence in distribution.
\end{proposition}


\section{Examples of Random Matrix Ensembles}
\label{Section:Examples}

In this section we gather some examples of random matrix
ensembles satisfying A0, A1 and A2. In particular, Example~2
shows that block Gaussian matrices fall under the framework
of our main results.

\subsection{Example 1: Gaussian Unitary Ensemble}

Let $(X_N)_N$ be the (normalized) Gaussian Unitary Ensemble
(GUE), i.e., for each $N\in\mathbb{N}$, $X_N$ is an $N\times
N$ self-adjoint random matrix such that $\{X_N(i,j) : 1\leq
i \leq j\leq N\}$ are independent random variables with
$X_N(i,i) \sim {\mathcal N}_\mathbb{R}(0,N^{-1})$ and
$X_N(i,j) \sim {\mathcal N}_\mathbb{C}(0,N^{-1})$ ($i\neq
j$). In this case:
\begin{itemize}
	\item[0)] $(X_N)_N$ has a second-order limit
          distribution. See Theorem~3.1 in
          \cite{MingoSpeicher2006}.
	
	\item[1)] For all $\epsilon>0$, there exists $C>0$ such that
	\begin{equation}
	\label{eq:LargeDeviationGUE}
		\mathbb{P}(\|X_N\| > 2(1+\epsilon)) \leq 2C
                \exp\left(- \frac{2\epsilon^2}{C} N\right).
	\end{equation}
	In particular, A1 is satisfied for every $M>2$. See
        (1.4) in \cite{LedouxRider2010}.
	
	\item[2)] If $f:\mathbb{R}\to\mathbb{C}$ is differentiable, then
	\begin{equation}
	\label{eq:A2GUE}
		\textnormal{Var}(\textnormal{Tr}(f(X_N))) \leq \|f'\|_\infty^2.
	\end{equation}
See Proposition~2.1.8 in \cite{PasturShcherbina2011}. This
shows that $(X_N)_N$ satisfies A2 with $K=1$.
\end{itemize}

As an additional comment, the second-order Cauchy transform
of this ensemble is given by
\begin{equation}
\label{eq:SOCauchyTransformSOF3}
	G_2(z,w) = \frac{G'(z)G'(w)}{[G(z)-G(w)]^2} - \frac{1}{(z-w)^2},
\end{equation}
where $\displaystyle G(z) = \frac{z-\sqrt{z^2-4}}{2}$. See
(7) in \cite{CollinsMingoSniadySpeicher2007}.  In
Theorem~3.1.1 in \cite{PasturShcherbina2011}, it was
established that
\begin{align}
\label{eq:SOCauchyTransformPasturShcherbina}
	\lim_{N\to\infty} &
        \textnormal{Cov}(\textnormal{Tr}((z-X_N)^{-1}),\textnormal{Tr}((w-X_N)^{-1}))
        \\ & = \frac{1}{2(z-w)^2}
        \left(\frac{zw-4}{\sqrt{z^2-4}\sqrt{w^2-4}} -
        1\right).\notag
\end{align}
Theorem~\ref{Theorem:ConvergenceSOCauchyTransform} leads to
the equality of the right hand sides of
\eqref{eq:SOCauchyTransformSOF3} and
\eqref{eq:SOCauchyTransformPasturShcherbina}. Given the
relative simplicity of these expressions, it can be shown
directly that they are actually equal. (Exercise for the
reader.)

We would like to point out that the Wishart/Laguerre
ensemble has a second-order limit distribution and satisfies
A1 and A2. See \cite{MingoNica2004},
\cite{dumitriu_edelman_2006}, Theorem 3.5 in
\cite{MingoSpeicher2006}, Theorem~2 in
\cite{LedouxRider2010}, and Proposition~7.2.1 in
\cite{PasturShcherbina2011}. As with the right hand sides of
\eqref{eq:SOCauchyTransformSOF3} and
\eqref{eq:SOCauchyTransformPasturShcherbina},
Theorem~\ref{Theorem:ConvergenceSOCauchyTransform}
establishes the non-trivial fact that the free probability
theory and the random matrix theory expressions for the
second-order Cauchy transform of this ensemble are equal.

\subsection{Example 2: Block Gaussian Matrices}
\label{Section:ExampleBlockGaussianMatrices}

Let $A_1,\ldots,A_r$ be $d\times d$ self-adjoint
matrices. Assume that $X_N^{(1)},\ldots,X_N^{(r)}$ are
independent GUE matrices as in Example~1. The $dN\times dN$
random matrix
\begin{equation*}
	X_N = \sum_{k=1}^r A_k\otimes X_N^{(k)}
\end{equation*}
is called a block Gaussian matrix. In this case:
\begin{itemize}
	\item[0)] Recall that $\textnormal{Tr}(A\otimes B) =
          \textnormal{Tr}(A)\textnormal{Tr}(B)$. Hence
	\begin{equation*}
		\textnormal{Tr}(X_N^n) =
                \sum_{k_1,\ldots,k_n=1}^r
                \textnormal{Tr}(A_{k_1} \cdots A_{k_n})
                \textnormal{Tr}(X_N^{(k_1)} \cdots
                X_N^{(k_n)}).
	\end{equation*}
	Since the $r^{th}$ cumulant, $k_r$, is $r$-linear,
        Theorem~3.1 in \cite{MingoSpeicher2006} readily
        shows that $(X_N)_N$ has a second-order limit
        distribution.
	
	\item[1)] A routine computation shows that, for every $M>0$,
	\begin{equation*}
		\mathbb{P}(\|X_N\|>M) \leq r
                \mathbb{P}(\left\|X_N^{(1)}\right\|
                >\frac{M}{r\max_k \|A_k\|}).
	\end{equation*}
	In particular, if we take $M_0=4r\max_k \|A_k\|$,
        then \eqref{eq:LargeDeviationGUE} implies that
	\begin{equation*}
		\mathbb{P}(\|X_N\|>M_0) \leq 2rC \exp\left(-\frac{2}{C} N\right).
	\end{equation*}
	Thus A1 is satisfied with $M=M_0$.
	
	\item[2)] If $f:\mathbb{R}\to\mathbb{C}$ is differentiable, then
	\begin{equation*}
		\textnormal{Var}(\textnormal{Tr}(f(X_N)))
                \leq r^2 \left\|\sum_{k=1}^r A_k^2\right\|^2
                \|f'\|_\infty^2.
	\end{equation*}
	See Proposition~4.7 in
        \cite{HaagerupThorbjornsen2005}. This shows that
        block Gaussian matrices satisfy A2.
\end{itemize}

As mentioned in the introduction, the second-order Cauchy
transform of block Gaussian matrices was recently found in
\cite{diaz2020global}, see equation (5).

\subsection{Example 3: $A+UBU^*$}

For each $N\in\mathbb{N}$, let $U_N$ be an $N\times N$ Haar
unitary. Assume that $(A_N)_N$ and $(B_N)_N$ are
self-adjoint non-random matrix ensembles such that their
eigenvalue distributions converge in distribution and
$\displaystyle T:=\sup_{N\in\mathbb{N}}
\max\{\|A_N\|,\|B_N\|\}$ is finite. Let $X_N = A_N + U_N B_N
U_N^*$. In this case:
\begin{itemize}
	\item[0)] $(X_N)_N$ has a second-order limit
          distribution. See Theorem~1 in
          \cite{MingoSniadySpeicher2007}.
	\item[1)] For all $N\in\mathbb{N}$, $\|X_N\| \leq
          2T$. In particular $X_N$ satisfies A1 with $M=2T$.
	\item[2)] If $f:\mathbb{R}\to\mathbb{C}$ is differentiable, then
	\begin{equation*}
		\textnormal{Var}(\textnormal{Tr}(f(X_N))) \leq 4T^2\|f'\|_\infty^2.
	\end{equation*}
Thus $X_N$ satisfies A2 with $K=4T^2$. See page 303 in
\cite{PasturShcherbina2011}.
\end{itemize}

For notational convenience, let $\mu_A$ and $\mu_B$ be the
limiting eigenvalue distributions of $(A_N)_N$ and
$(B_N)_N$, respectively. By
Theorem~\ref{Theorem:ConvergenceSOCauchyTransform},
\begin{equation*}
G_2(z,w) = \lim_{N\to\infty}
\textnormal{Cov}(\textnormal{Tr}((z-X_N)^{-1}),\textnormal{Tr}((w-X_N)^{-1})).
\end{equation*}
In the notation of Chapter~3 in \cite{MingoSpeicher2017},
Theorem~10.2.1 and (10.2.30) in \cite{PasturShcherbina2011}
imply that
\begin{equation*}
G_2(z,w) = \frac{\partial^2}{\partial z \partial w} \log
\frac{\omega_A(z)-\omega_A(w)}{z-w}
\frac{\omega_B(z)-\omega_B(w)}{F(z)-F(w)},
\end{equation*}
where $\omega_A$ and $\omega_B$ are the so-called
subordination functions and
$F(z)=1/G_{\mu_A\boxplus\mu_B}(z)$. Revisiting this
expression for the second-order Cauchy transform is of
interest in view of the recent developments in random matrix
theory based on the subordination functions, e.g.,
\cite{BelinschiBercoviciCapitaineFevrier2014,
  BelinschiMaiSpeicher2015, Shlyakhtenko2015,
  CapitaineDonatiMartin2016, BaoErdosSchnelli2017}.

\section{Analytic Second-Order Non-Commutative Probability Spaces}
\label{Section:AnalyticsSONCPS}

In order to motivate the definition of analytic second-order
NCPS below, recall the integral representation for the
asymptotic covariance $\rho$:
\begin{equation}
\label{eq:IntegralRepresentationRho}
	\rho(f,g) \coloneqq \int_{-M}^{M} \int_{-M}^{M} f'(x) g'(y) \mathrm{d}u(x,y).
\end{equation}	 

Observe that the integral representation in
\eqref{eq:IntegralRepresentationRho} only depends on
$f\vert_{M}$ and $g\vert_{M}$. Since $f\vert_{M}\in
C^{1}([-M,M])$ for every polynomially bounded $f\in
C^1(\bR)$, in the sequel we can restrict our attention to
the function space $C^{1}([-M,M])$.

We define the concept of analytic second-order NCPS
motivated by the integral representation of the asymptotic
covariance mapping associated to a random matrix ensemble in
\eqref{eq:IntegralRepresentationRho}.

\begin{definition}
A second-order NCPS $(\mathcal{A},\varphi,\rho)$, with
$\mathcal{A} = C^{1}([-M,\ab M])$, for some $M>0$, is called
\emph{analytic} if
\begin{itemize}
	\item[$(a)$] there exists a probability measure
          $\mu$ on $[-M,M]$ such that, for all
          $f\in\mathcal{A}$,
	\begin{equation*}
		\varphi(f) = \int f(x) \mathrm{d}\mu(x);
	\end{equation*}
	\item[$(b)$] there exists a bounded Fr\'{e}chet
          variation function $u:[-M,M]^2 \ab\to\mathbb{R}$
          such that, for all $f,g\in\mathcal{A}$,
	\begin{equation}
	\label{eq:IntregralRepresentationRho}
		\rho(f,g) = \int_{-M}^{M} \int_{-M}^{M} f'(x) g'(y) \mathrm{d} u(x,y).
	\end{equation}
\end{itemize}
\end{definition}

Observe that Theorem~\ref{Theorem:BilinearFunctional} $(c)$
implies that the second-order NCPS associated to a random
matrix ensemble satisfying A0, A1, and A2 is indeed
analytic. Nonetheless, it is unknown to the authors if every
analytic second-order NCPS arise in this manner, i.e., if
for every $u:[-M,M]^2\to\mathbb{R}$ of bounded Fr\'{e}chet
variation there exists a random matrix ensemble whose
asymptotic covariance mapping satisfies
\eqref{eq:IntegralRepresentationRho}.

The Fr\'{e}chet representation theorem establishes that if
$u:[-M,M]^2 \ab\to\mathbb{R}$ has bounded Fr\'{e}chet
variation, then the bilinear functional
$\Phi:C([-M,M])^{2}\to\mathbb{C}$ defined by
\begin{equation*}
	\Phi(f,g) = \int_{-M}^{M} \int_{-M}^{M} f(x) g(y) \mathrm{d} u(x,y).
\end{equation*}
is continuous in each argument with respect to the supremum norm. Since
\begin{equation*}
\rho = \Phi \circ \left(\frac{\mathrm{d}}{\mathrm{d}
  x}\times\frac{\mathrm{d}}{\mathrm{d} y}\right),
\end{equation*}
we conclude that the unbounded nature of $\rho$ comes from
the differential operator $\frac{\mathrm{d}}{\mathrm{d}
  x}\times\frac{\mathrm{d}}{\mathrm{d} y}$. This observation
immediately leads to the following criterion for the
exchange of limit and $\rho$.

\begin{lemma}\label{lemma:continuity_of_rho}
Let $(\mathcal{A},\varphi,\rho)$ be an analytic second-order
NCPS with $\cA = C^1([ -M, M])$. Suppose $\{ f_n \}_{n =
  1}^\infty$ and $\{ g_n \}_{n=1}^\infty$ are in $C^1([ -M,
  M])$.  If $\{f_{n}'\}_n$ converges uniformly to $f'$ and
$\{g_{n}'\}_n$ converges uniformly to $g'$, then
\begin{equation*}
	\rho(f,g) = \lim_{n\to\infty} \rho(f_{n},g_{n}).
\end{equation*}
\end{lemma}

Now we turn our attention to the analytic version of the
second-order Cauchy transform. For each
$z\in\mathbb{C}\setminus[-M,M]$, let
$r_{z}:[-M,M]\to\mathbb{C}$ be defined by
\begin{equation*}
	r_{z}(x) = \frac{1}{z-x}.
\end{equation*}
Observe that $r_{z}$ is a differentiable function, i.e., $r_{z}\in C^{1}([-M,M])$.

\begin{definition}
\label{def:soncps}
Let $(\mathcal{A},\varphi,\rho)$ be an analytic second-order
NCPS. We define the (analytic) second-order Cauchy
transform \[G_{2}:(\mathbb{C}\setminus[-M,M])^{2}\ab\to\mathbb{C}\]
as
\begin{equation}
\label{eq:DefinitionAnalyticSOCauchyTransform}
	G_{2}(z,w) \coloneqq \rho(r_{z},r_{w}).
\end{equation}
\end{definition}

Observe that \eqref{eq:DefinitionAnalyticSOCauchyTransform}
and the integral representation of $\rho$ in
\eqref{eq:IntregralRepresentationRho} imply that, for all
$z,w\in\mathbb{C}\setminus[-M,M]$,
\begin{equation}
\label{eq:IntegralRepresentationSOCauchyTransform}
G_{2}(z,w) = \int_{-M}^{M} \int_{-M}^{M} \frac{1}{(z-x)^2}
\frac{1}{(w-y)^2} \mathrm{d} u(x,y).
\end{equation}
Indeed, the second-order Cauchy transform of many random
matrix ensembles in the literature have an integral
representation of the form
\eqref{eq:IntegralRepresentationSOCauchyTransform}, see,
e.g., \cite[Eq. (1.7)]{BaiSilverstein2004} and
\cite[Eq. (2.9)]{diaz_jaramillo_pardo_2020}.

In the next proposition we show that the second-order Cauchy
transform is an analytic function whose power series
expansion at infinity coincides with the generating function
in \eqref{eq:IntroSOCauchyTransform}.

\begin{proposition}
If $(\mathcal{A},\varphi,\rho)$ is an analytic second-order
NCPS with $\cA = C^1([-M, M])$, then $G_{2}$ is an analytic
function on $(\mathbb{C}\setminus[-M,M])^2$ such that, for
all $|z|,|w|>M$,
\begin{equation*}
G_{2}(z,w) = \sum_{m,n\geq0} \frac{\rho(x^m,x^n)}{z^{m+1}w^{n+1}}.
\end{equation*}
\end{proposition}

\begin{proof}
For $|z| > M$ the series, considered as a function of $x$,
$\sum_{n=1}^\infty \frac{n x^{n-1}}{z^{n+1}}$ converges
uniformly to $(z - x)^{-2}$ on $[-M, M]$. Likewise, for $|w|
> M$, the series $\sum_{n=1}^\infty \frac{n y^{n-1}}{w^n}$
converges uniformly to $(w - y)^{-2}$ on $[-M, M]$. By Lemma
\ref{lemma:continuity_of_rho} we have
\begin{align*}
G_2(z, w) & = \int_{-M}^M \int_{-M}^M (z -x)^{-2} (w -
y)^{-2} \, du(x,y)\\ & = \sum_{m, n \geq 1} \frac{m
  n}{z^{m+1}w^{n+1}} \int_{-M}^M \int_{-M}^M x^{m-1} y^{n-1}
\, du(x,y) \\ & = \sum_{m, n \geq 1} \frac{\rho(x^m,
  x^n)}{z^{m+1}w^{n+1}},
\end{align*}
as required.
\end{proof}

Note that the previous proposition generalizes
Theorem~\ref{Theorem:ConvergenceSOCauchyTransform}. Indeed,
if the second-order NCPS $(\mathcal{A},\varphi,\rho)$ is
associated to a random matrix ensemble satisfying A0, A1,
and A2, then the previous proposition is a simple
consequence of
Theorem~\ref{Theorem:ConvergenceSOCauchyTransform}. In a
similar spirit, the following theorem generalizes
Theorem~\ref{Theorem:FluctuationsLinearStatistics}.

\begin{theorem}
Assume that $(\mathcal{A},\varphi,\rho)$ is an analytic
second-order NCPS. If $f,g\in C^{1}([-M,M])$ extend
analytically to a complex domain $\Omega \supset [-M,M]$,
then
\begin{equation}\label{eq:int_rep_rho}
\rho(f,g) = \frac{1}{(2\pi i)^2} \int_{\mathcal{C}}
\int_{\mathcal{C}} f(z) g(w) G_{2}(z,w) \mathrm{d} z
\mathrm{d} w,
\end{equation}
where $\mathcal{C}\subset\Omega$ is a positively oriented
simple closed contour enclosing $[-M,M]$.
\end{theorem}

\begin{proof}
On the right hand side of (\ref{eq:int_rep_rho}) we may
interchange the integral for $G_2(z, w)$ given by
(\ref{eq:IntegralRepresentationSOCauchyTransform}) with the
two contour integrals, as all the integrals are over compact
sets. Then for $x \in [-M, M]$ we have $f'(x) = \frac{1}{2
  \pi i} \int_C f(z)(z - x)^{-2} \, \mathrm dz$; likewise
$g'(y) = \frac{1}{2 \pi i} \int_C g(w)(w - y)^{-2} \,
\mathrm dw$ for $y \in [-M, M]$. With these substitutions we
have the left hand side of (\ref{eq:int_rep_rho}).
\end{proof}

\section{Proofs of Theorems \ref{Theorem:BilinearFunctional}, \ref{Theorem:ConvergenceSOCauchyTransform}, and \ref{Theorem:FluctuationsLinearStatistics}}
\label{Section:DeferredProofs}

The following notation will be used through the rest of this
section. Assume that a self-adjoint random matrix ensemble
$(X_N)_N$ is given. For two Borel measurable functions
$f,g:\mathbb{R}\to\mathbb{C}$, we define
\begin{align*}
\varphi_N(f) &= \frac{1}{N}
\mathbb{E}(\textnormal{Tr}(f(X_N))) \quad \text{ and }
\\ \rho_N(f,g) &=
\textnormal{Cov}(\textnormal{Tr}(f(X_N)),\textnormal{Tr}(g(X_N))).
\end{align*}
Recall that $M$ is the constant from A1. For
$f:\mathbb{R}\to\mathbb{C}$, we define
$f_M:\mathbb{R}\to\mathbb{C}$ by $f_M(x) = f(x)
\mathbbm{1}_{|x|\leq M}$. The proofs of our main results
rely heavily on the following truncation lemmas.

\begin{lemma}
\label{Lemma:Truncation}
Let $(X_N)_N$ be a self-adjoint random matrix ensemble
satisfying A0. Assume that $f,g:\mathbb{R}\to\mathbb{C}$ are
Borel measurable functions.
\begin{itemize}	
	\item[$(a)$] If $f$ and $g$ are bounded, then, for
          all $N\in\mathbb{N}$,
	\begin{equation}
	\label{eqa:TruncationInqA}
		\left|\rho_N(f,g) - \rho_N(f_M,g_M)\right|
                \leq 4 \|f\|_\infty \|g\|_\infty N^2
                \mathbb{P}(\|X_N\|>M)^{1/4}.
	\end{equation}
	\item[$(b)$] If $f$ and $g$ are polynomially
          bounded, then there exists $K_{f,g}>0$ such that,
          for all $N\in\mathbb{N}$,
	\begin{equation}
	\label{eqa:TruncationInqB}
		\left|\rho_N(f,g) - \rho_N(f_M,g_M)\right|
                \leq K_{f,g} N^2
                \mathbb{P}(\|X_N\|>M)^{1/4}.
	\end{equation}
\end{itemize}
\end{lemma}

\begin{proof}
Let $\lambda_1\leq\cdots\leq\lambda_N$ be the eigenvalues of
$X_N$. Observe that, for any Borel measurable function
$h:\mathbb{R}\to\mathbb{C}$,
\begin{equation*}
	\textnormal{Tr}(h(X_N)) = \textnormal{Tr}(h_M(X_N))
        + \sum_{i=1}^N
        h(\lambda_i)\mathbbm{1}_{|\lambda_i|>M}.
\end{equation*}
In particular, we have that $\displaystyle
\rho_N\left(f,g\right) - \rho_N(f_M,g_M) = \text{I} +
\text{II} + \text{III}$, where
\begin{align*}
\text{I} &= \sum_{j=1}^N
\textnormal{Cov}(\textnormal{Tr}(f_M(X_N)),
g(\lambda_j)\mathbbm{1}_{|\lambda_j|>M}),\\ \text{II} &=
\sum_{i=1}^N
\textnormal{Cov}(f(\lambda_i)\mathbbm{1}_{|\lambda_i|>M},
\textnormal{Tr}(g_M(X_N))),\\ \text{III} &= \sum_{i,j=1}^N
\textnormal{Cov}(f(\lambda_i)\mathbbm{1}_{|\lambda_i|>M},
g(\lambda_j)\mathbbm{1}_{|\lambda_j|>M}).
\end{align*}
By the Cauchy-Schwarz inequality, we have that
\begin{equation*}
|\text{I}| \leq \textnormal{Var}(\textnormal{Tr}
(f_M(X_N)))^{1/2} \sum_{j=1}^N
\textnormal{Var}(g(\lambda_j)\mathbbm{1}_{|\lambda_j|>M})^{1/2}.
\end{equation*}
Another application of the Cauchy-Schwarz inequality shows
that, for any Borel measurable function
$h:\mathbb{R}\to\mathbb{C}$,
\begin{equation*}
	|\textnormal{Tr}(h_M(X_N))|^2 \leq N \sum_{i=1}^N
        |h_M(\lambda_i)|^2 \leq N
        \textnormal{Tr}(|h(X_N)|^2).
\end{equation*}
Therefore, 
\begin{equation*}
	\textnormal{Var}(\textnormal{Tr}(f_M(X_N)) \leq
        \mathbb{E}(|\textnormal{Tr}(f_M(X_N))|^2) \leq N^2
        \varphi_N(|f|^2),
\end{equation*}   
and hence
\begin{equation*}
	|\text{I}| \leq N \varphi_N(|f|^2)^{1/2}
        \sum_{j=1}^N
        \textnormal{Var}(g(\lambda_j)\mathbbm{1}_{|\lambda_j|>M})^{1/2}.
\end{equation*}
H\"{o}lder's inequality implies that, for all $j\in\{1,\ldots,N\}$,
\begin{align*}
	\textnormal{Var}(g(\lambda_j)\mathbbm{1}_{|\lambda_j|>M}) & 
	\leq \mathbb{E}(|g(\lambda_j)|^2\mathbbm{1}_{|\lambda_j|>M}) \\
	 & \mbox{} \leq \mathbb{E}(|g(\lambda_j)|^4)^{1/2} \mathbb{P}(|\lambda_j|>M)^{1/2}.
\end{align*}
Observe that $\mathbb{P}(|\lambda_j|>M) \leq
\mathbb{P}(\|X_N\|>M)$. Thus,
\begin{equation*}
	|\text{I}| \leq N \, \varphi_N(|f|^2)^{1/2}
        \mathbb{P}(\|X_N\|>M)^{1/4} \, \sum_{j=1}^N
        \mathbb{E}(|g(\lambda_j)|^4)^{1/4}.
\end{equation*}
Another application of the generalized mean inequality shows that
\begin{equation*}
	\sum_{j=1}^N \mathbb{E}(|g(\lambda_j)|^4)^{1/4} \leq
        N \varphi_N(|g|^4)^{1/4}.
\end{equation*}
Therefore,
\begin{equation*}
	|\text{I}| \leq \varphi_N(|f|^2)^{1/2}
        \varphi_N(|g|^4)^{1/4} N^2
        \mathbb{P}(\|X_N\|>M)^{1/4}.
\end{equation*}
Mutatis mutandis, it is possible to show that
\begin{align*}
|\text{II}| &\leq \varphi_N(|f|^4)^{1/4}\,
|\varphi_N(|g|^2)^{1/2} \, N^2
|\mathbb{P}(\|X_N\|>M)^{1/4},\\ \text{III}| &\leq
|\varphi_N(|f|^4)^{1/4}\, \varphi_N(|g|^4)^{1/4} \, N^2
|\mathbb{P}(\|X_N\|>M)^{1/2}.
\end{align*}
The three inequalities for $|$I$|$, $|$II$|$, and $|$III$|$
imply that
\begin{equation*}
	\left|\rho_N(f,g) - \rho_N(f_M,g_M)\right| \leq
        K_{f,g,N} \, N^2 \mathbb{P}(\|X_N\|>M)^{1/4},
\end{equation*}
where
\begin{equation*}
	K_{f,g,N} = \left[\varphi_N(|f|^2)^{1/2} +
          \varphi_N(|f|^4)^{1/4}\right]
        \left[\varphi_N(|g|^2)^{1/2} +
          \varphi_N(|g|^4)^{1/4}\right].
\end{equation*}
Part $(a)$ is an easy consequence of the fact that, for any
Borel measurable function $h:\mathbb{R}\to\mathbb{C}$ and
$p>0$,
\begin{equation*}
	\varphi_N(|h|^p)^{1/p} = \left(\frac{1}{N}
        \sum_{i=1}^N \mathbb{E}(|h(\lambda_i)|^p)
        \right)^{1/p} \leq \|h\|_\infty.
\end{equation*}
If $f$ and $g$ are polynomially bounded, then there exists a
polynomial $q:\mathbb{R}\to\mathbb{R}$ such that
\begin{equation*}
	\max\{|f(x)|^2,|f(x)|^4,|g(x)|^2,|g(x)|^4\} \leq q(x),
\end{equation*}
for all $x\in\mathbb{R}$. In particular, for all
$N\in\mathbb{N}$,
\begin{equation*}
	K_{f,g,N} \leq \left[\varphi_N(q)^{1/2} +
          \varphi_N(q)^{1/4}\right]^2.
\end{equation*}
By assumption A0, the limit $\displaystyle \lim_{N\to\infty}
\varphi_N(q)$ exists. Taking
\begin{equation*}
	K_{f,g} = \sup \left\{\left[\varphi_N(q)^{1/2} +
          \varphi_N(q)^{1/4}\right]^2 :
        N\in\mathbb{N}\right\},
\end{equation*}
Part $(b)$ now follows.
\end{proof}

We recall that $f|_M$ denotes the restriction of
$f:\mathbb{R}\to\mathbb{C}$ to the interval $[-M,M]$. Note
that $f|_M$ is a function on $[-M,M]$, while $f_M$ is a
function on $\mathbb{R}$. For a function
$f:\mathbb{R}\to\mathbb{C}$, we let
\begin{equation*}
	\tilde{f}(x) = \begin{cases} f(x) &
          x\in[-M,M],\\ f(M) + f'(M) (x-M) e^{-\alpha (x-M)}
          & x>M,\\ f(-M) + f'(-M)(x+M) e^{\alpha (x+M)} &
          x<-M,\end{cases}
\end{equation*}
where $\alpha=\|f'\|_M / (e\|f\|_M)$ with $\|h\|_M = \sup \{
|h(x)| : x\in[-M,M]\}$. By construction, $\tilde{f}_M =
(\tilde f)_M = f_M$. It is not hard to verify that if
$f:\mathbb{R}\to\mathbb{C}$ is such that $f|_M\in
C^1([-M,M])$, then
\begin{equation*}
	\tilde{f}\in C^1(\mathbb{R}), \quad
        \|\tilde{f}\|_\infty \leq 2 \|f\|_M, \quad \text{
          and } \quad \|\tilde{f}'\|_\infty = \|f'\|_M.
\end{equation*}

\begin{lemma}
\label{Lemma:FiniteDimensionContinuity}
Let $(X_N)_N$ be a self-adjoint random matrix ensemble
satisfying \az{} and \at{}, and let $K$ be the constant in
\at{}. Assume that $f,g:\mathbb{R}\to\mathbb{C}$ are Borel
measurable functions such that $f|_M,g|_M\in C^1([-M,M])$.
\begin{itemize}
	\item[$(a)$] If $f$ and $g$ are bounded, then, for
          all $N\in\mathbb{N}$,
	\begin{equation*}
		|\rho_N(f,g)| \leq 20 \|f\|_\infty
                \|g\|_\infty N^2 \mathbb{P}(\|X_N\| >
                M)^{1/4} + K \|f'\|_M \|g'\|_M.
	\end{equation*}
	\item[$(b)$] If $f$ and $g$ are polynomially
          bounded, then there exists $K_{f,g}>0$ such that,
          for all $N\in\mathbb{N}$,
	\begin{equation*}
		|\rho_N(f,g)| \leq K_{f,g} N^2
                \mathbb{P}(\|X_N\| > M)^{1/4} + K \|f'\|_M
                \|g'\|_M.
	\end{equation*}
\end{itemize}
\end{lemma}

\begin{proof}
Recall that, by construction $\tilde{h}_M = h_M$. In
particular,
\begin{align}
\label{eq:Splitrhofg}\lefteqn{
	|\rho_N(f,g)| \leq |\rho_N(f,g) - \rho_N(f_M,g_M)| }
\\ &\mbox{} + |\rho_N(\tilde{f}_M,\tilde{g}_M) -
\rho_N(\tilde{f},\tilde{g})| +
|\rho_N(\tilde{f},\tilde{g})|.
\end{align}
If $f$ and $g$ are bounded, Part $(a)$ of
Lemma~\ref{Lemma:Truncation} implies that
\begin{align*}
	|\rho_N(f,g)| & \leq 4 (\|f\|_\infty \|g\|_\infty +
        \|\tilde{f}\|_\infty \|\tilde{g}\|_\infty) N^2
        \mathbb{P}(\|X_N\|>M)^{1/4} \\ & \qquad\mbox{} +
        |\rho_N(\tilde{f},\tilde{g})|.
\end{align*}
Since $\|\tilde{f}\|_\infty \leq 2 \|f\|_M \leq 2
\|f\|_\infty$, we conclude that
\begin{equation*}
	|\rho_N(f,g)| \leq 20 \|f\|_\infty \|g\|_\infty N^2
        \mathbb{P}(\|X_N\|>M)^{1/4} +
        |\rho_N(\tilde{f},\tilde{g})|.
\end{equation*}
By A2, we have that $|\rho_N(\tilde{f},\tilde{g})| \leq K
\|\tilde{f}'\|_\infty \|\tilde{g}'\|_\infty = K \|f'\|_M
\|g'\|_M$. Part $(a)$ follows.

If $f$ and $g$ are polynomially bounded,
\eqref{eq:Splitrhofg} and Part $(b)$ of
Lemma~\ref{Lemma:Truncation} imply that
\begin{equation*}
	|\rho_N(f,g)| \leq (K_{f,g}'+4\|\tilde{f}\|_\infty
        \|\tilde{g}\|_\infty) N^2
        \mathbb{P}(\|X_N\|>M)^{1/4} +
        |\rho_N(\tilde{f},\tilde{g})|,
\end{equation*}
for some $K_{f,g}'>0$. Let $K_{f,g} = K_{f,g}'+16\|f\|_M
\|g\|_M$. In particular,
\begin{align*}
|\rho_N(f,g)| &\leq K_{f,g} N^2 \mathbb{P}(\|X_N\|>M)^{1/4}
+ |\rho_N(\tilde{f},\tilde{g})|\\ &\leq K_{f,g} N^2
\mathbb{P}(\|X_N\|>M)^{1/4} + K \|f'\|_M \|g'\|_M,
\end{align*}
where the last inequality follows from A2 and the fact that
$\|\tilde{f}'\|_\infty = \|f'\|_M$.
\end{proof}

\begin{proof}[\bf Proof of Theorem~\ref{Theorem:BilinearFunctional}]
We start by constructing the bilinear mapping $\rho$. Let
$p,q\in\mathbb{C}[x]$. We define
\begin{equation*}
	\rho(p|_M,q|_M) = \lim_{N\to\infty} \rho_N(p,q).
\end{equation*}
Note that, by A0, the limit in the previous equation
exists. By Part $(b)$ of
Lemma~\ref{Lemma:FiniteDimensionContinuity}, we have that
\begin{equation*}
	|\rho_N(p,q)| \leq K_{p,q} N^2
        \mathbb{P}(\|X_N\|>M)^{1/4} + K \|p'\|_M \|q'\|_M,
\end{equation*}
for some $K_{p,q}>0$. Taking limits, A1 implies then
\begin{equation*}
	|\rho(p|_M,q|_M)| \leq K \|p'\|_M \|q'\|_M.
\end{equation*}
In particular, we have 
\[
|\rho(p|_M, q|_M)| \leq K \| p \| \| q \|,
\]
where $\lVert p \rVert = \lvert p(0) \rvert + \lVert p'
\rVert_{M}$ as introduced in
Notation~\ref{notation:our_space_of_functions}. Since the
polynomials without constant term are dense in $C^1([-M,
  M])^\circ$ with respect to the norm $\| \cdot \|$, it is a
standard procedure to extend $\rho$ continuously (in each
argument) to $C^1([-M,M])^\circ$ with the same bound. We
extend $\rho$ to $\bC \oplus C^1([-M, M])^\circ$ by making
$\rho$ vanish on $\bC \oplus 0$; again without increasing
the norm. Let $\rho:C^1([-M,M]) \times C^1([-M,M]) \to
\mathbb{C}$ denote this extension. This proves
$(a)$. Since we already had $\rho(p,1) =
\rho(1,q) = 0$ for polynomials $p$ and $q$, this extension
is consistent with the definition of $\rho$.

We prove $(b)$ next. Let $f:\mathbb{R}\to\mathbb{C}$ be a
polynomially bounded function with $f|_M\in C^1([-M,M])$ and
$q\in\mathbb{C}[x]$. For any polynomial $p\in\mathbb{C}[x]$,
we set $\rho_N(f,q) - \rho(f|_M,q|_M) = \text{I} + \text{II}
+ \text{III}$, where
\begin{align*}
	\text{I} &= \rho_N(f,q) -\rho_N(p,q) \text{, }
        \text{II} = \rho_N(p,q) - \rho(p|_M,q|_M) \text{,
        }\\ &\qquad\qquad\text{III} = \rho(p|_M,q|_M) -
        \rho(f|_M,q|_M).
\end{align*}
By Part $(b)$ of
Lemma~\ref{Lemma:FiniteDimensionContinuity}, we have that
\begin{align*}
	|\text{I}| \leq K_{f-p,q} N^2 \mathbb{P}(\|X_N\| >
        M)^{1/4} + K \|(f-p)'\|_M \|q'\|_M,
\end{align*}
for some $K_{f-p,q}>0$. Let $\epsilon>0$. By the density of
the polynomials in $C^1([-M,M])$ with respect to the
$C^1$-norm and the continuity of $\rho$ with respect to the
same norm, there exists $p_0\in\mathbb{C}[x]$ such that
\begin{equation*}
	\|(f-p_0)'\|_M < \frac{\epsilon}{6K\|q'\|_M} \quad
        \text{ and } |\text{III}| < \frac{\epsilon}{3}.
\end{equation*}
By construction, $\rho_N(p_0,q) \to \rho((p_0)|_M,q|_M)$ as
$N\to\infty$. Combined with A1, this implies that there
exists $N_0\in\mathbb{N}$ such that for all $N\geq N_0$
\begin{equation*}
	|\text{II}| < \frac{\epsilon}{3} \quad \text{ and }
        K_{f-p_0,q} N^2 \mathbb{P}(\|X_N\| > M)^{1/4} <
        \frac{\epsilon}{6}.
\end{equation*}
Therefore, for all $N>N_0$, $|\rho_N(f,q) - \rho(f|_M,q|_M)|
< \epsilon$, i.e.,
\begin{equation*}
	\lim_{N\to\infty} \rho_N(f,q) = \rho(f|_M,q|_M).
\end{equation*}
The previous equation and a similar argument show that
\begin{equation*}
	\lim_{N\to\infty} \rho_N(f,g) = \rho(f|_M,g|_M)
\end{equation*}
for all polynomially bounded functions
$f,g:\mathbb{R}\to\mathbb{C}$ with $f|_M,g|_M\in
C^1([-M,M])$. This proves Part $(b)$.

In order to prove Part $(c)$, let
$\Phi:C([-M,M])^2\to\mathbb{C}$ be given by $\Phi(f,g) =
\rho(F,G)$, where
\begin{equation*}
	F(x) = \int_{0}^x f(t) \mathrm{d} t \quad \text{ and
        } \quad G(x) = \int_{0}^x g(t) \mathrm{d} t.
\end{equation*}
Note that $\Phi$ is bilinear and, by Part $(a)$,
\begin{equation*}
	|\Phi(f,g)| = |\rho(F,G)| \leq K \|f\|_M \|g\|_M.
\end{equation*}
By Theorem \ref{thm:frechet_II}, there exists
$u:[-M,M]^2\to\mathbb{R}$ of bounded Fr\'{e}chet variation
such that, for all $f,g\in C([-M,M])$,
\begin{equation}
\label{eq:ProofIntegralRepresentation}
	\Phi(f,g) = \int_{-M}^{M} \int_{-M}^{M} f(x) g(x) \mathrm{d} u(x,y).
\end{equation}
Then $\rho(f, g) = \Phi(f', g') = \displaystyle
\int_{-M}^{M} \int_{-M}^{M} f'(x) g'(y) \, \mathrm d u(x,
y)$, as claimed.
\end{proof}

For $z\in\mathbb{C}$, we let
$r_z:\mathbb{C}\setminus\{z\}\to\mathbb{C}$ be given by
$\displaystyle r_z(x) = \frac{1}{z-x}$ and $r_{z,M} =
(r_z)_M$. Note that, for all
$z\in\mathbb{C}\setminus[-M,M]$, the function
$r_{z,M}:\mathbb{R}\to\mathbb{C}$ satisfies
\begin{equation*}
	\|r_{z,M}\|_\infty = d(z)^{-1} \quad \text{ and }
        \quad \|r_{z,M}'\|_M = d(z)^{-2},
\end{equation*}
where $d(z) := \inf \{ |z-x| : x\in[-M,M]\}$.  We define
$G_{2,M}^{(N)}:(\mathbb{C}\setminus[-M,M])^2\to\mathbb{C}$
by
\begin{equation}
\label{eq:DefG2Truncated}
	G_{2,M}^{(N)}(z,w) = \rho_N(r_{z,M},r_{w,M}).
\end{equation}
It is not hard to verify that $G_{2,M}^{(N)}$ is analytic. Indeed,
\begin{align*}
\rho_N( r_{z,M}, r_{w, M}) &= \sum_{i,j, = 1}^N \E(
\mathbbm{1}_{|\lambda_i|\leq M} (z- \lambda_{i})^{-1}
\mathbbm{1}_{|\lambda_j|\leq M} (w - \lambda_{j})^{-1})
\\ &\qquad\quad\mbox{} - \E(\mathbbm{1}_{|\lambda_i|\leq
  M}(z- \lambda_{i})^{-1}) \E( \mathbbm{1}_{|\lambda_j|\leq
  M}(w - \lambda_{j})^{-1}).
\end{align*}
Observe that the function $z \mapsto \E( \mathbbm{1}_{|\lambda_i|\leq M}
(z- \lambda_{i})^{-1})$ is the Cauchy transform of the
random variable $\lambda_{i}$ conditioned on $|\lambda_i|
\leq M$ times the probability of this event, and is analytic
on $\bC \setminus [-M, M]$, c.f. \cite[Lemma
  3.2]{MingoSpeicher2017}. If $|z' - z| \leq \frac{1}{2}
d(z)$, then for $w \in \bC \setminus [-M, M]$ we have
\begin{align*}\lefteqn{
\E(\mathbbm{1}_{|\lambda_i|\leq M}(z- \lambda_{i})^{-1} 
\mathbbm{1}_{|\lambda_j|\leq M}(w - \lambda_{j})^{-1}) }\\
& = -
\sum_{n=0}^\infty \E\bigg( 
\mathbbm{1}_{|\lambda_i|\leq M}\mathbbm{1}_{|\lambda_j|\leq M}
\frac{(w - \lambda_{j})^{-1}}{(\lambda_{i} - z)^{n+1}}\bigg) (z' - z)^n
\end{align*}
and the convergence is uniform on $\{ z' \in \bC \mid |z' -
z| < \frac{1}{2} d(z)\}$, c.f. \cite[the proof of Lemma
  3.2]{MingoSpeicher2017}. Thus $G_{2,M}^{(N)}$ is analytic
on $(\bC \setminus [-M, M])^2$. Moreover, it satisfies the
following boundedness property.

\begin{lemma}
If $(X_N)_N$ is a self-adjoint random matrix ensemble
satisfying $A0$, $A1$, and $A2$, then for every
$z_0,w_0\in\mathbb{C}\setminus[-M,M]$ there exists
$\delta>0$ such that
\begin{equation*}
	\sup \left\{\big|G_{2,M}^{(N)}(z,w)\big| :
        |z-z_0|<\delta,|w-w_0|<\delta,N\in\mathbb{N}\right\}
        < \infty.
\end{equation*}
\end{lemma}

\begin{proof}
By Part $(a)$ of
Lemma~\ref{Lemma:FiniteDimensionContinuity}, for all
$z,w\in\mathbb{C}\setminus[-M,M]$,
\begin{align*}
	|G_{2,M}^{(N)}(z,w)| &\leq 20 \|r_{z,M}\|_\infty
        \|r_{w,M}\|_\infty N^2 \mathbb{P}(\|X_N\| > M)^{1/4}
        \\ &\qquad \mbox{} + K \|r_{z,M}'\|_M
        \|r_{w,M}'\|_M\\ &= 20 d(z)^{-1} d(w)^{-1} N^2
        \mathbb{P}(\|X_N\| > M)^{1/4} + K d(z)^{-2}
        d(w)^{-2}.
\end{align*}
Let $\delta=\frac{1}{2} \min(d(z_0),d(w_0))$. Note that for
all $z,w\in\mathbb{C}$ such that $|z-z_0|<\delta$ and
$|w-w_0|<\delta$, we have that $d(z)>\delta$ and
$d(w)>\delta$. Assumption $A1$ implies that $\{N^2
\mathbb{P}(\|X_N\|>M)^{1/4} : N\in\mathbb{N}\}$ is
bounded. The lemma now follows.
\end{proof}

\begin{lemma}
\label{Lemma:KeyConvergenceLemma}
If $(X_N)_N$ is a self-adjoint random matrix ensemble
satisfying $A0$, $A1$, and $A2$, then the family
$\left\{G_{2,M}^{(N)} : N\in\mathbb{N}\right\}$ converges
uniformly in compact subsets of
$(\mathbb{C}\setminus[-M,M])^2$ and, for all $|z|,|w|>M$,
\begin{equation*}
	\lim_{N\to\infty} G_{2,M}^{(N)}(z,w) = G_2(z,w).
\end{equation*}
\end{lemma}

\begin{proof}
By definition,
\begin{equation*}
	G_{2,M}^{(N)}(z,w) =
        \textnormal{Cov}\left(\sum_{i=1}^N
        \frac{\mathbbm{1}_{|\lambda_i|\leq
            M}}{z-\lambda_i},\sum_{j=1}^N
        \frac{\mathbbm{1}_{|\lambda_j|\leq
            M}}{w-\lambda_j}\right).
\end{equation*}
If $|z|,|w|>M$, then we have that
\begin{equation*}
	(z-\lambda_i)^{-1} = \sum_{m\geq0}
  \frac{\lambda_i^m}{z^{m+1}} \quad \text{ and } \quad
  (w-\lambda_j)^{-1} = \sum_{n\geq0}
  \frac{\lambda_j^n}{w^{n+1}}.
\end{equation*}
In particular, for such $z$ and $w$, we have that
\begin{equation*}
	G_{2,M}^{(N)}(z,w) =
        \textnormal{Cov}\left(\sum_{i=1}^N \sum_{m\geq0}
        \frac{\lambda_i^m \mathbbm{1}_{|\lambda_i|\leq
            M}}{z^{m+1}},\sum_{j=1}^N \sum_{n\geq0}
        \frac{\lambda_j^n \mathbbm{1}_{|\lambda_j|\leq
            M}}{w^{n+1}}\right).
\end{equation*}
For each $k\in\mathbb{N}$, let
$\pi^{(k)}:\mathbb{R}\to\mathbb{R}$ denote the function
given by $\pi^{(k)}(x) = x^k$. With this notation, we can
rewrite the previous equation as
\begin{equation*}
	G_{2,M}^{(N)}(z,w) =
        \textnormal{Cov}\left(\sum_{m\geq0}
        \frac{\textnormal{Tr}(\pi^{(m)}_M(X_N))}{z^{m+1}},\sum_{n\geq0}
        \frac{\textnormal{Tr}(\pi^{(n)}_M(X_N))}{w^{n+1}}\right).
\end{equation*}
Since $|\textnormal{Tr}(\pi^{(k)}_M(X_N))|\leq NM^k$, a
routine application of Tonelli-Fubini theorem implies that
\begin{equation*}
	G_{2,M}^{(N)}(z,w) = \sum_{m,n\geq0}
        \frac{\rho_N(\pi^{(m)}_M,\pi^{(n)}_M)}{z^{m+1}w^{n+1}}.
\end{equation*}
Note that $\pi^{(k)}_M$ satisfies that $\displaystyle
\|\pi^{(k)}_M\|_\infty = M^k$ and $\displaystyle
\|(\pi^{(k)}_M)'\|_M = k M^{k-1}$. By Part $(a)$ of
Lemma~\ref{Lemma:FiniteDimensionContinuity}, we obtain
\begin{equation*}
	|\rho_N(\pi^{(m)}_M,\pi^{(n)}_M)| \leq 20 M^{m+n}
        N^2 \mathbb{P}(\|X_N\| > M)^{1/4} + mnKM^{m+n-2}.
\end{equation*}
By A1, the set $\left\{N^2 \mathbb{P}(\|X_N\|>M)^{1/4} :
N\in\mathbb{N}\right\}$ is bounded. Hence,
\begin{equation*}
	|\rho_N(\pi^{(m)}_M,\pi^{(n)}_M)| \leq mnBM^{m+n},
\end{equation*}
for some constant $B>0$ independent of $m$ and $n$. The
previous inequality and the dominated convergence theorem
imply that
\begin{equation*}
	\lim_{N\to\infty} G_{2,M}^{(N)}(z,w) =
        \sum_{m,n\geq0} \lim_{N\to\infty}
        \frac{\rho_N(\pi^{(m)}_M,\pi^{(n)}_M)}{z^{m+1}w^{n+1}},
\end{equation*}
for all $|z|,|w|>M$. Part $(a)$ of
Lemma~\ref{Lemma:Truncation} implies that
\begin{equation*}
	\lim_{N\to\infty} \rho_N(\pi^{(m)}_M,\pi^{(n)}_M) =
        \rho(\pi^{(m)}|_M,\pi^{(n)}|_M) = \alpha_{m,n}.
\end{equation*}
In other words, for all $|z|,|w|>M$,
\begin{equation}
\label{eq:LimitG2MNG2}
	\lim_{N\to\infty} G_{2,M}^{(N)}(z,w) = G_2(z,w).
\end{equation}
By the previous lemma, the family
$\left\{G_{2,M}^{(N)}\right\}_N$ is locally
bounded. Therefore, Montel's theorem
\cite[pp. 33]{Scheidemann2005} and \eqref{eq:LimitG2MNG2}
imply that $\left\{G_{2,M}^{(N)}\right\}_N$ converges
uniformly on compact sets to an analytic function on
$(\mathbb{C}\setminus[-M,M])^2$.
\end{proof}

For notational simplicity, for each
$z,w\in\mathbb{C}\setminus\mathbb{R}$, we let
\begin{equation*}
	G_2^{(N)}(z,w) = \rho_N(r_z,r_w) =
        \textnormal{Cov}(\textnormal{Tr}((z-X_N)^{-1}),
        \textnormal{Tr}((w-X_N)^{-1})).
\end{equation*}

\begin{proof}[\bf Proof of Theorem~\ref{Theorem:ConvergenceSOCauchyTransform}]
The previous lemma states that $G_2(z,w)= \lim_{N}
G_{2,M}^{(N)}(z,w)$ for all $|z|,|w|>M$. Also, it
establishes that $\left\{G_{2,M}^{(N)}\right\}_N$ converges
uniformly in compact subsets of
$(\mathbb{C}\setminus[-M,M])^2$ to an analytic function,
which extends $G_2$ to the domain $\{ (z, w) \mid |z|, |w| >
M\}$; we also call this extension $G_2$. By Part $(a)$ of
Lemma~\ref{Lemma:Truncation}, we have that for $z, w \not\in
\bR$,
\begin{equation*}
	\left|G_2^{(N)}(z,w) - G_{2,M}^{(N)}(z,w)\right|
        \leq 4 \|r_z\|_\infty \|r_w\|_\infty N^2
        \mathbb{P}(\|X_N\|>M)^{1/4}.
\end{equation*}
Since $\|r_z\|_\infty\leq|\Im z|^{-1}$ for all
$z\in\mathbb{C}\setminus\mathbb{R}$, $A1$ implies that
\begin{equation*}
	\lim_{N\to\infty} G_2^{(N)}(z,w) = \lim_{N\to\infty}
        G_{2,M}^{(N)}(z,w) = G_2(z,w),
\end{equation*}
as required.
\end{proof}

\begin{proof}[\bf Proof of Theorem~\ref{Theorem:FluctuationsLinearStatistics}]
Whenever $|\lambda_i|\leq M$, Cauchy integral formula
implies that
\begin{equation*}
	f(\lambda_i) = \frac{1}{2\pi i} \int_\mathcal{C}
        \frac{f(z)}{z-\lambda_i} \mathrm{d} z.
\end{equation*}
In particular, we have that
\begin{equation*}
	\rho_N(f_M,g_M) =\kern-3pt \sum_{i,j=1}^N
        \textnormal{Cov}\Big(\frac{1}{2\pi i}
        \int_\mathcal{C} \frac{f(z)
          \mathbbm{1}_{|\lambda_i|\leq M}}{z-\lambda_i}
        \mathrm{d} z,\frac{1}{2\pi i} \int_\mathcal{C}
        \frac{g(w) \mathbbm{1}_{|\lambda_j|\leq
            M}}{w-\lambda_j} \mathrm{d} w\Big).
\end{equation*}
A routine application of the Tonelli-Fubini theorem shows then that
\begin{align*}
	\rho_N(f_M,g_M) &= \sum_{i,j=1}^N \frac{1}{(2\pi
          i)^2} \int_\mathcal{C} \int_\mathcal{C} f(z) g(w)
        \textnormal{Cov}\Big(\frac{\mathbbm{1}_{|\lambda_i|\leq
            M}}{z-\lambda_i},\frac{\mathbbm{1}_{|\lambda_j|\leq
            M}}{w-\lambda_j}\Big) \mathrm{d} z \mathrm{d}
        w\\ &= \frac{1}{(2\pi i)^2} \int_\mathcal{C}
        \int_\mathcal{C} f(z) g(w) G_{2,M}^{(N)}(z,w)
        \mathrm{d} z \mathrm{d} w.
\end{align*}
By Lemma~\ref{Lemma:KeyConvergenceLemma},
$\left\{G_{2,M}^{(N)}\right\}_N$ converges uniformly to (the
extension of) $G_2$ on the compact set
$\mathcal{C}\times\mathcal{C}$. Therefore, by the dominated
convergence theorem,
\begin{equation*}
	\lim_{N\to\infty} \rho_N(f_M,g_M) = \frac{1}{(2\pi
          i)^2} \int_\mathcal{C} \int_\mathcal{C} f(z) g(w)
        G_2(z,w) \mathrm{d} z \mathrm{d} w.
\end{equation*}
By Part $(b)$ of Theorem~\ref{Theorem:BilinearFunctional},
we conclude that
\begin{equation*}
	\lim_{N\to\infty} \rho_N(f,g) = \rho(f|_M,g|_M) =
        \lim_{N\to\infty} \rho_N(f_M,g_M).
\end{equation*}
Equation \eqref{eq:ConvergenceLinearStatistics} follows.
\end{proof}

\section{Summary and Concluding Remarks}
\label{Section:ConcludingRemarks}

Under general assumptions, we established the existence and
boundedness of the asymptotic covariance mapping $\rho$
(Theorem~\ref{Theorem:BilinearFunctional}). Also, we showed
that the second-order Cauchy transform $G_{2}$ admits an
analytic extension which is equal to the limit of the
covariance of resolvents
(Theorem~\ref{Theorem:ConvergenceSOCauchyTransform}). Furthermore,
we showed that the asymptotic covariance mapping $\rho$
could be recovered from the second-order Cauchy transform
(Theorem~\ref{Theorem:FluctuationsLinearStatistics}). For
random matrix ensembles having a second-order limit
distribution, we showed that the fluctuations of the linear
statistics of a real smooth (test) function $f$ are
asymptotically Gaussian with mean zero and variance
$\rho(f,f)$ (Proposition~\ref{Proposition:CLT}). In addition
we proved that if $(X_N)_N$ is a self-adjoint random matrix
ensemble satisfying A0, A1, and A2, then there exists
$u:\mathbb{R}^2\to\mathbb{R}$ of bounded Fr\'{e}chet
variation such that $\textnormal{Supp}(u)$ is compact and
\begin{equation*}
	G_2(z,w) = \int_{\mathbb{R}^2} \frac{1}{(z-x)^2}
        \frac{1}{(w-y)^2} \mathrm{d} u(x,y),
\end{equation*}
where the integral above is in the sense of Fr\'{e}chet
\cite{Frechet1915,Hildebrandt1963}. In this case, we say
that $u$ is the second-order analytic distribution of
$(X_N)_N$.

Let $\mathcal{F}$ be the set of functions on $\mathbb{R}^2$
of bounded Fr\'{e}chet variation which are the second-order
analytic distribution of a random matrix ensemble satisfying
A0, A1, and A2. In this setting, the characterization of
$\mathcal{F}$ is of interest.

\vspace{10pt}

\noindent{\bf Q1.} Given $u:\mathbb{R}^2\to\mathbb{R}$ of
bounded Fr\'{e}chet variation, does it belong to
$\mathcal{F}$?

\vspace{10pt}

Assume that $u_1,u_2\in\mathcal{F}$ with associated random
matrix ensembles $(X_N)_N$ and $(Y_N)_N$, respectively. For
each $N\in\mathbb{N}$, let $U_N$ be an $N\times N$ Haar
unitary matrix independent of $X_N$ and $Y_N$.

\vspace{10pt}

\noindent{\bf Q2.} The random matrix ensemble
$(X_N+U_NY_NU_N^*)_N$ satisfies A0 and A1, does it satisfy
also A2?

\vspace{10pt}

If the previous question has an affirmative answer, we can
define the second-order free additive convolution of $u_1$
and $u_2$ as the second-order analytic distribution of
$X_N+U_NY_NU_N^*$. This would define a binary operation on
$\mathcal{F}$:
\begin{equation*}
	\boxplus_2 : \mathcal{F}\times\mathcal{F} \to
        \mathcal{F},
\end{equation*}
similarly as the $\boxplus$-operation does in the
first-order setting.

In bi-free probability theory there are Cauchy transforms of
the form
\begin{equation*}
	G(z,w) = \int_{\mathbb{R}^2} \frac{1}{z-x}
        \frac{1}{w-y} \mathrm{d} u(x,y),
\end{equation*}
for some compactly supported measure $u$. Note that
\begin{equation}
\label{eq:SOBF}
	\frac{\partial^2}{\partial z \partial w} G(z,w) =
        \int_{\mathbb{R}^2} \frac{1}{(z-x)^2}
        \frac{1}{(w-y)^2} \mathrm{d} u(x,y) = G_2(z,w).
\end{equation}
In view of the expression (see
\cite[Eq. (5.20)]{MingoSpeicher2017}), where $F(z) =
1/G(z)$,
\begin{equation*}
	G_2(z,w) = G'(z)G'(w)R(G(z),G(w)) +
        \frac{\partial^2}{\partial z \partial w} \log
        \frac{F(z)-F(w)}{z-w},
\end{equation*}
equation \eqref{eq:SOBF} suggests a potential connection
between second-order free probability theory and bi-free
probability theory.

\vspace{10pt}

\noindent{\bf Q3.} If we assume that $R(z,w)\equiv0$, is it
possible to provide new examples in bi-free probability
theory using second-order free probability ones?

\vspace{10pt}

In particular, the previous question provides some
motivation to study the case $R(z,w)\equiv0$, which indeed
has attracted some attention in the past.

\section*{Acknowledgements}
MD would like to thank Arturo Jaramillo and Roland Speicher
for fruitful discussions while preparing this paper. Also,
MD would like to thank Malors Espinosa for his keen comments
concerning Lemma~\ref{Lemma:ApproximationRunge}.

\begin{appendix}
\section{Proof of Proposition~\ref{Proposition:CLT}}
\label{Appendix:ProofPropositionCLT}

The following lemma is an easy consequence of Runge's
theorem and the Schwarz Reflection Principle. For a set
$S\subset\mathbb{C}$, we let $S^*=\{z\in\mathbb{C} :
\overline{z}\in S\}$.

\begin{lemma}
\label{Lemma:ApproximationRunge}
Let $\Omega\subset\mathbb{C}$ be a domain. Assume that
$\mathcal{K}=\mathcal{K}^*$ is a compact subset of $\Omega$
whose complement is connected. If $f:\Omega\to\mathbb{C}$ is
an analytic function such that
$f(\mathbb{R})\subset\mathbb{R}$, then there exist
polynomials $(p_k)_{k\in\mathbb{N}}\subset\mathbb{R}[z]$
such that
\begin{equation*}
	\lim_{k\to\infty} \|p_k - f\|_\mathcal{K} = 0 \quad
        \text{ and } \quad \lim_{k\to\infty}
        \|(p_k-f)'\|_\mathcal{K} = 0,
\end{equation*}
where $\|g\|_\mathcal{K} = \sup \{|g(z)| : z\in\mathcal{K}\}$.
\end{lemma}

\begin{proof}[Proof of Proposition~\ref{Proposition:CLT}] Let real valued $f \in \mathcal{S}$ be given. 
Let $\cK\subset\bC$ be a compact set that contains
$[-M,M]$. Let $\Omega \subset \bC$ be an open disc
containing $\cK$. By Lemma \ref{Lemma:ApproximationRunge}
there exist real polynomials $(p_k)_k$ such that
\begin{equation*}
	\lim_{k\to\infty} \|p_k - f\|_\mathcal{K} = 0 \quad
        \text{ and } \quad \lim_{k\to\infty}
        \|(p_k-f)'\|_\mathcal{K} = 0.
\end{equation*}
For notational simplicity, for all $k\in\mathbb{N}$ and
$N\in\mathbb{N}$, we let
\begin{equation*}
	Z_N^{(k)} = \textnormal{Tr}(p_k(X_N)) -
        \mathbb{E}(\textnormal{Tr}(p_k(X_N))).
\end{equation*}
Clearly $k_1\left(Z_N^{(k)}\right)=0$. From Part ii) in the
definition of second-order limit distribution, it is
immediate to see that, for all $r\geq3$,
\begin{equation*}
	\lim_{N\to\infty} k_r\left(Z_N^{(k)},\ldots,Z_N^{(k)}\right) = 0.
\end{equation*}
Since the Gaussian distribution on $\mathbb{R}$ is
characterized by its moments, and hence by its cumulants, we
conclude from Part (b) of
Theorem~\ref{Theorem:BilinearFunctional} that
$\{\displaystyle Z_N^{(k)}\}_N$ converges in distribution
(i.e., $Z_N^{(k)} \Rightarrow Z^{(k)}$) to $Z^{(k)} \sim
{\mathcal N}_\mathbb{R}(0,\sigma_k^2)$, where $\sigma_k^2 =
\rho(p_k|_M,p_k|_M)$. Let $\sigma^2 = \rho(f|_M,f|_M)$. By
Part a) of the same theorem,
\begin{align*}
|\sigma_k^2 - \sigma^2| &\leq  |\rho(p_k|_M,p_k|_M) - \rho(p_k|_M,f|_M)|\\
& \qquad \mbox{}  + |\rho(p_k|_M,f|_M) - \rho(f|_M,f|_M)|\\
&\leq K (\|p_k'\|_\mathcal{K} + \|f'\|_\mathcal{K}) \|(p_k-f)'\|_\mathcal{K}.
\end{align*}
In particular, $\sigma_k^2 \to \sigma^2$ and hence $Z^{(k)}
\Rightarrow Z \sim {\mathcal N}_\mathbb{R}(0,\sigma^2)$ as
$k\to\infty$.

For all $N\in\mathbb{N}$, we let $Z_N =
\textnormal{Tr}(f(X_N)) -
\mathbb{E}(\textnormal{Tr}(f(X_N)))$. Let
$g:\mathbb{R}\to\mathbb{R}$ be a bounded Lipschitz
function. Note that, for all $N\in\mathbb{N}$ and all
$k\in\mathbb{N}$,
\begin{equation*}
	\mathbb{E}(g(Z_N)) - \mathbb{E}(g(Z)) = \alpha_{k,N}
        + \beta_{k,N} + \gamma_{k},
\end{equation*}
where
\begin{align*}
	\alpha_{k,N} &= \mathbb{E}\left(g(Z_N) -
        g\left(Z_N^{(k)}\right)\right),\\ \beta_{k,N} &=
        \mathbb{E}\left(g\left(Z_N^{(k)}\right) -
        g\left(Z^{(k)}\right)\right),\\ \gamma_{k} &=
        \mathbb{E}\left(g\left(Z^{(k)}\right) - g(Z)\right).
\end{align*}
Let $L_g$ be a Lipschitz constant for $g$. A routine
computation shows that
\begin{equation*}
	|\alpha_{k,N}| \leq L_g \,
        \mathbb{E}\left(\left|Z_N-Z_N^{(k)}\right|\right)
        \leq L_g \textnormal{Var}(Z_N-Z_N^{(k)})^{1/2}.
\end{equation*}
Observe that
\begin{equation*}
	\textnormal{Var}(Z_N-Z_N^{(k)}) =
        \textnormal{Var}(\textnormal{Tr}((p_k-f)(X_N))) =
        \rho_N(p_k-f,p_k-f).
\end{equation*}
By Part b) of Lemma~\ref{Lemma:FiniteDimensionContinuity},
there exists $C_k>0$ such that, for all $N\in\mathbb{N}$,
\begin{equation*}
	|\rho_N(p_k-f,p_k-f)| \leq C_k N^2
        \mathbb{P}(\|X_N\| > M)^{1/4} + K
        \|(p_k-f)'\|_\mathcal{K}^2.
\end{equation*}
In particular, we have that
\begin{equation*}
	|\alpha_{k,N}| \leq L_g \left(C_k N^2
        \mathbb{P}(\|X_N\|>M)^{1/4} + K
        \|(p_k-f)'\|_\mathcal{K}^2\right)^{1/2}.
\end{equation*}
Let $\epsilon>0$. Since $Z^{(k)} \Rightarrow Z$, we have
that $\mathbb{E}(g\left(Z^{(k)}\right)) \to
\mathbb{E}(g(Z))$ as $k\to\infty$. Let $k_0\in\mathbb{N}$ be
such that
\begin{equation*}
	|\gamma_{k_0}| =
        \left|\mathbb{E}(g\left(Z^{(k_0)}\right)) -
        \mathbb{E}(g(Z))\right| \leq \frac{\epsilon}{3}
        \quad \text{ and } \quad
        \|(p_{k_0}-f)'\|_\mathcal{K}^2 \leq
        \frac{\epsilon^2}{18KL_g^2}.
\end{equation*}
Since $N^8 \mathbb{P}(\|X_N\|>M) \to 0$ and $Z_N^{(k_0)}
\Rightarrow Z^{(k_0)}$ as $N\to\infty$, there exists
$N_0\in\mathbb{N}$ such that for all $N>N_0$
\begin{align*}\lefteqn{
C_{k_0} N^2 \mathbb{P}(\|X_N\|>M)^{1/4} \leq
\frac{\epsilon^2}{18L_g^2}} \\
  &
\text{ and } \quad |\beta_{k_0,N}| =
\left|\mathbb{E}\left(g\left(Z_N^{(k_0)}\right)\right) -
\mathbb{E}\left(g\left(Z^{(k_0)}\right)\right)\right|\leq
\frac{\epsilon}{3}.
\end{align*}
Therefore, $\displaystyle |\mathbb{E}(g(Z_N)) -
\mathbb{E}(g(Z))| \leq \epsilon$ for all $N>N_0$, i.e.,
\begin{equation*}
	\lim_{N\to\infty} \mathbb{E}(g(Z_N)) =
        \mathbb{E}(g(Z)).
\end{equation*}
The Portmanteau lemma implies then that $Z_N \Rightarrow
Z$. In other words,
\begin{equation*}
	\textnormal{Tr}(f(X_N)) -
        \mathbb{E}(\textnormal{Tr}(f(X_N))) \Rightarrow
               {\mathcal N}_\mathbb{R}(0,\sigma^2),
\end{equation*}
as we claimed.
\end{proof}
\end{appendix}

\end{document}